\documentclass[a4paper]{scrartcl} 

\usepackage[utf8]{inputenc}
\usepackage[T1]{fontenc}
\usepackage{lmodern} 
\usepackage{amssymb, amsmath, amsfonts, amsthm} 
\usepackage{mathtools} \mathtoolsset{showonlyrefs} 
\usepackage{mathrsfs} 
\usepackage{enumerate}
\usepackage{color}
\usepackage{graphicx} 
\usepackage{mdwlist} 
\usepackage{paralist} 
\usepackage{cite}
\usepackage{ulem} 
\usepackage{microtype} 
\usepackage{hyperref} \usepackage[all]{hypcap} 

\usepackage{float} 

	\hypersetup{colorlinks=false} 

	\let\oldleft\left
	\let\oldright\right
	\renewcommand{\left}{\mathopen{}\mathclose\bgroup\oldleft}
	\renewcommand{\right}{\aftergroup\egroup\oldright}

	\theoremstyle{definition}
	
	\newtheorem{Remark}[equation]{Remark}

	\theoremstyle{plain}
	\newtheorem{Theorem}[equation]{Theorem}
	\newtheorem{Lemma}[equation]{Lemma}

	\let\op\operatorname

	\renewcommand{\epsilon}{\varepsilon}

	\newcommand{\Z}{{\mathbb Z}}
	
	\newcommand{\R}{{\mathbb R}}

	\newcommand{\BigO}{{O}}

	\newcommand{\E}[1]{\mathbb E\left[#1\right]} 

	\newcommand{\eg}{e.g.\ }

	\newcommand{\acos}{\cos\inv}
	\newcommand{\asin}{\sin\inv}
	\newcommand{\atan}{\tan\inv}

	\DeclareMathOperator*{\union}{{\bigcup}} 

	\newcommand{\defeq}{\mathrel{\mathop:}=} 
	\newcommand{\inv}{^{-1}}

	\newcommand{\abs}[1]{\left|#1\right|}
	\newcommand{\norm}[1]{\left\|#1\right\|}

	\newcommand{\dfn}[1]{\textbf{#1}}

	\newcommand{\dd}[1]{\mathop{d#1}}

	\newcommand{\dmu}[2]{\mathop{d#1}(#2)}

	\newcommand{\pp}[1]{\subsection{#1}}

\newcommand{\x}{a}
\newcommand{\y}{b}
\newcommand{\z}{c}
\newcommand{\vx}{v_x}
\newcommand{\vy}{v_y}
\newcommand{\vz}{v_z}
\newcommand{\fp}[1]{\left\{#1\right\}_1}

\title{On the free path length distribution for linear motion in
  an $n$-dimensional box}

\author{Samuel Holmin \and
	Pär Kurlberg \and
	Daniel Månsson}

\begin{document}

\maketitle

\begin{abstract}
  We consider the distribution of free path lengths, or the distance
  between consecutive bounces of random particles, in an
  $n$-dimensional rectangular box.  If each particle travels a distance
  $R$, then, as $R \to \infty$ the free path lengths coincides with
  the distribution of the length of the intersection of a random line
  with the box (for a natural ensemble of random lines) and we
  give an explicit formula
  (piecewise real analytic)
  for the probability density function in dimension two and three.

  In dimension two we also consider a closely related model where each
  particle is allowed to bounce $N$ times, as $N \to \infty$, and give
  an explicit (again piecewise real analytic)
  formula for its probability density function.

  Further, in both models we can recover the side lengths of the box
  from the location of the discontinuities of the probability density
  functions.
\end{abstract}


\section{Introduction}
We consider billiard dynamics on a rectangular domain, i.e., point
shaped ``balls'' moving with linear motion with specular reflections at
the boundary, and similarly for rectangular box shaped domains in
three dimensions.  We wish to determine the distribution of free path
lengths of ensembles of trajectories defined by selecting a starting
point and direction at random.

The question seems quite natural and interesting on its own, but we
mention that it originated from the study of electromagnetic fields in
``reverberation chambers'' under the assumption of highly directional
antennas \cite{mansson-pers-comm}.  Briefly, the connection is as
follows (we refer to the forthcoming paper \cite{applied-paper} for
more details): given an ideal highly directional antenna and a highly
transient signal,
then the wave pulse dynamics is essentially the same as a point shaped billiard
ball traveling inside a chamber, with specular reflection at the
boundary.  Signal loss is dominated by (linear) ``spreading'' of the
electromagnetic field and by absorption occurring at each interaction
(``bounce'') with the walls.
The first simple model we use in this paper neglects absorption
effects, and models signal loss from spreading by simply terminating
the motion of the ball after it has travelled a certain large
distance.  The second model only takes into account signal loss from
absorption, and completely neglects spreading; here the motion is
terminated after the ball has bounced a certain number of times.

We remark that the distribution of free path lengths is very well
studied in the context of the Lorentz gas --- here a point particle
interacts with hard spherical obstacles, either placed randomly, or
regularly on Euclidean lattices; recently quasicrystal configurations
have also been studied (cf. \cite{bunimovich-sinai,
  boldrighini-bunimovich-sinai-boltzmann-lorenz,spohn-lorentz,
  golse-wennberg-lorentz-gas, bourgain-golse-wennberg-lorentz-gas,
  marklof-strombergsson-boltzman-grad-annals2011,
  marklof-strombergsson-CMP-quasicrystal,
  wennberg-quasicrystal,marklof-strombergsson-annals-2010}.)

Let $R>0$ be large and let a rectangular $n$-dimensional box
$K\subseteq \R^n$ be given, where $n\geq 2$. We send off a large number $M>0$ of particles,
each with a random initial position $p^{(i)}\in K$ chosen with respect
to a given probability measure $\mu$ on $K$, and each with a uniformly
random initial direction $v^{(i)}\in \mathbb S^{n-1}=\{x\in\R^{n}: \norm
x=1\}$, $i=1,\ldots,M$, for a total distance $R$ each. 
Each particle travels along straight lines, changing direction
precisely when it hits the boundary of the box, where it reflects
specularly. We record the distance 
travelled between each pair of consecutive bounces for each particle.
(Note in particular that we obtain more bounce lengths from some
particles than from others.) 
Let $X_{M,R}$ be the uniformly distributed random variable on
this finite set of bounce lengths of all the particles.
More precisely, a random sample of $X_{M,R}$ is obtained as follows:
first take a random i.i.d.\  sample of points (with respect to the
measure $\mu$) $p^{(1)}, \ldots, p^{(M)} \in K$, and a random sample of
directions $v^{(1)},\ldots,v^{(M)} \in \mathbb S^{n-1}$ (with respect to
the uniform measure).  Each pair $(p^{(i)},v^{(i)})$ then defines a
trajectory $T^{i}$ of length $R$, and each such trajectory gives rise
to a finite multiset $B^{i}$ of lengths between consecutive
bounces.  Finally, with $B = \union_{i=1}^{M} B^{i}$ denoting the
(multiset) union of bounce length multisets $B^{1},\ldots,B^{M}$, we
select an element of $B$ with the uniform distribution.  (That is, with
$1_{B}$ denoting the integer valued set indicator function for $B$,
and $B' = \{ x : 1_{B}(x) \geq 1 \}$ we select the element $b \in B'$
with probability $1_{B}(b)/\sum_{x \in B'} 1_{B}(x)$.)

We are interested in the distribution of $X_{M,R}$ for large $M$ and
$R$, and this turns out to be closely related to a model arising from
integral geometry.
Namely, let $\dd\ell$ denote the unique (up to a constant) translation-
and rotation-invariant measure on the set of directed lines $\ell$ in
$\R^n$, and consider the restriction of this measure to the set of
directed lines $\ell$ intersecting $K$, normalized such that it
becomes a probability measure. Denote by $X$ the random variable
$X\defeq \operatorname{length}(\ell\cap K)$ where $\ell$ is chosen at
random using this measure. 
\begin{Theorem}
\label{thm_n}
For any dimension $n\geq 2$, and for any distribution $\mu$ on the
starting points, the random variable $X_{M,R}$ converges in distribution
to the random variable $X$, as we take $R\to\infty$ followed by taking
$M\to\infty$, or vice versa.
\end{Theorem}

The mean free path length has a quite simple geometric interpretation. We have
  \begin{gather} 
    \E{X} = 2\pi \dfrac{|\mathbb S^{n-1}|}{|\mathbb S^n|}\dfrac{\operatorname{Vol}(K)}{\operatorname{Area}(K)}
     = 2\sqrt{\pi} \cdot \dfrac{\Gamma(\frac{n+1}2)}{\Gamma(\frac n2)}\dfrac{\operatorname{Vol}(K)}{\operatorname{Area}(K)} \label{santalo-formula}
  \end{gather}
  where $\operatorname{Area}(K)$ is the $(n-1)$-dimensional surface
  area of the box $K$, $\operatorname{Vol}(K)$ is the volume of the
  box $K$, $\Gamma$ is the gamma function, and where
  $|\mathbb S^{n-1}|=2\pi^{n/2}/\Gamma(n/2)$ is the $(n-1)$-dimensional surface area
  of the sphere $\mathbb S^{n-1}\subseteq \R^n$.
  The formula in \eqref{santalo-formula} has been proven in a more general setting earlier (see \eg formula (2.4) in \cite{chernov-mean-free-path-billiards}); for further details, see Section~\ref{sec:discussion}. For the convenience of the reader we give a short proof of formula \eqref{santalo-formula} in our setting in Section~\ref{sec:integral-geometry-mean-value}.

Throughout the paper, we will write $\op{pdf}_Z$ and $\op{cdf}_Z$ for
the probability density function and the cumulative distribution
function of $Z$, respectively, for random variables $Z$. 
We next give explicit formulas for the probability density function of $X$
in dimensions two and three.

\begin{Theorem}
\label{thm_2}
  For a box of dimension $n=2$ with side-lengths $\x\leq\y$, the probability density function of $X$ is given by
  \begin{gather}
  \operatorname{pdf}_X(t) = \frac1{\x+\y}\cdot 
    \begin{cases}
      1,&\text{if }t<\x,\y \\
      \dfrac{\x^2\y}{t^2\sqrt{t^2-\x^2}},&\text{if }\x<t<\y \\
      -1 + \dfrac1{t^2}\left( \dfrac{\x^2\y}{\sqrt{t^2-\x^2}} + \dfrac{\x\y^2}{\sqrt{t^2-\y^2}} \right),&\text{if }\x,\y<t.
    \end{cases}
  \end{gather}
  for $0<t<\sqrt{\x^2+\y^2}$.
\end{Theorem}

\begin{Remark}
  We note that the probability density function in Theorem \ref{thm_2}
  is analytic on all open subintervals of $(0, \sqrt{\x^2+\y^2})$ not
  containing $\x$ or $\y.$ Moreover, it is constant on the interval
  $(0,\min(\x,\y))$ and has singularities of type $(t-\x)^{-1/2}$ and
  $(t-\y)^{-1/2}$ just to the right of $\x$ and $\y$, respectively.
  See Figure~\ref{fig:easy-rectangle} for more details.  For an
  explanation of these singularities, see
  Remark~\ref{rem:singularity-explained}.
\end{Remark}

\begin{figure}[H]
\label{fig:easy-rectangle}
  \centering
  \includegraphics[width=13cm]{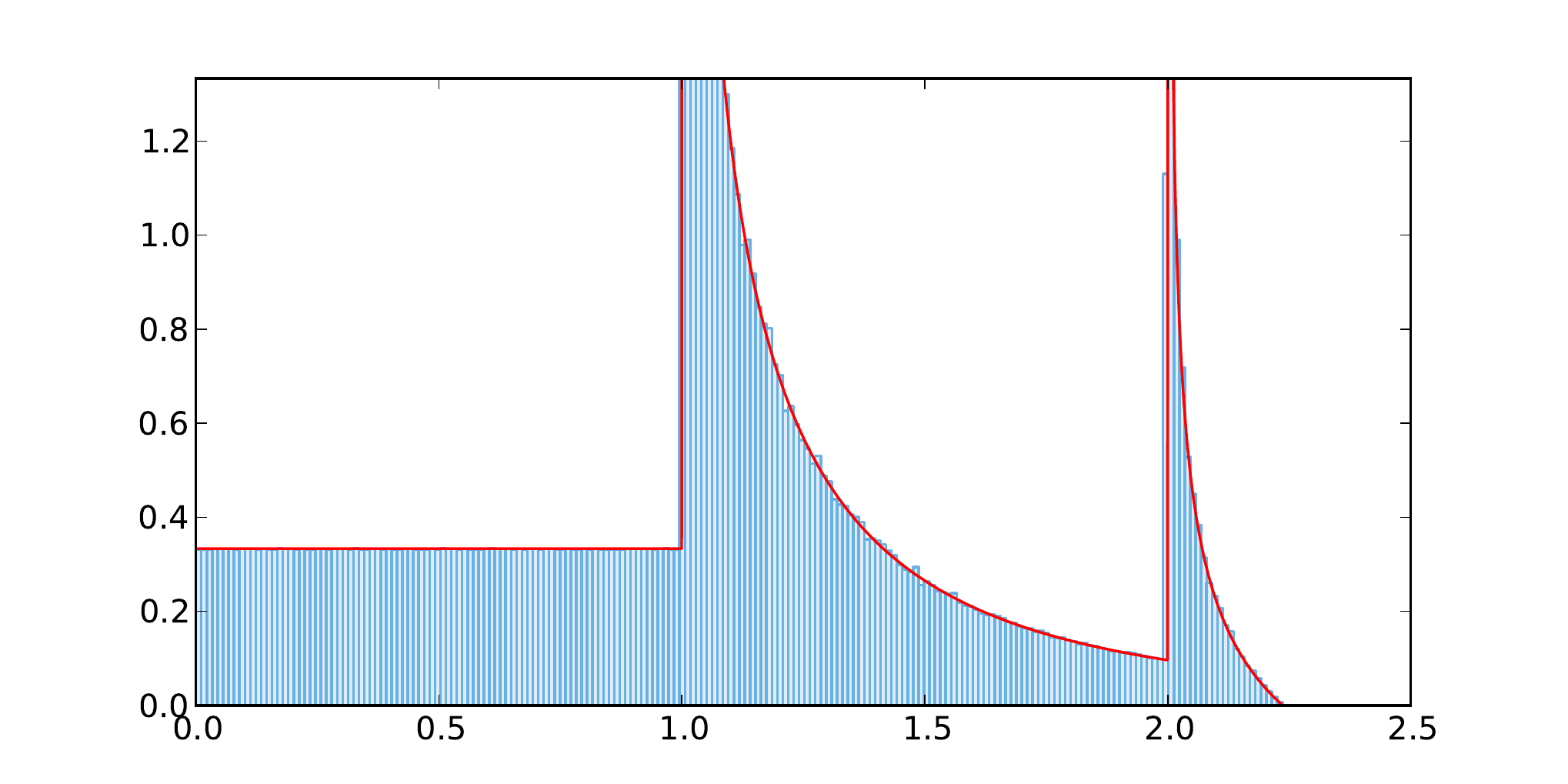}
  \caption{
    Simulation (blue histogram) vs explicit probability density
    function (red line) given by Theorem \ref{thm_2} for
    $(\x,\y)=(1,2)$. (Simulation used $10^5$ particles, each
    starting at the origin with a uniformly random direction, going
    for a total distance $1000$ each.)  The plot is cutoff at $y=1.3$
    since $\operatorname{pdf}_X(t)$ tends to infinity as $t \to 1^{+}$
    and $t \to 2^{+}$.
  }
\end{figure}

\begin{Theorem}
\label{thm_3}
  For a box of dimension $n=3$ with side-lengths $\x,\y,\z$, the probability density function of $X$ is given by
  \begin{gather}
    \operatorname{pdf}_X(t) = \dfrac{F(\x,\y,\z,t) + F(\y,\z,\x,t) + F(\z,\x,\y,t)}{3 \pi  t^3 (\x\y+\x\z+\y\z)}
  \end{gather}
  where $F$ is the piecewise-defined function given by
  \begin{gather}
    F(\x,\y,\z,t) = t^3(8\x-3t)
  \end{gather}
  for $0<t<\x$, and by
  \begin{gather}
    F(\x,\y,\z,t) =
          \left(6t^4-\x^4+6 \pi  \x^2 \y \z\right)-4(\y+\z)\sqrt{\abs{t^2-\x^2}}(\x^2+2t^2)
  \end{gather}
  for $\x<t<\sqrt{\x^2+\y^2}$,
  and by
  \begin{gather}
    F(\x,\y,\z,t) = 6 \pi  \x^2 \y \z +\y^4-3 t^4-6 \x^2 \y^2 + \\
            \sqrt{\abs{t^2-\x^2-\y^2} } 4\z\left( \x^2 + \y^2 +2 t^2\right)+\\
            +4\x\sqrt{\abs{t^2-\y^2}}(\y^2+2t^2)-12 \x^2 \y \z \cdot \arctan\left(\dfrac {\sqrt{\abs{t^2-\x^2-\y^2} }}\y\right)+\\
            -4\z\sqrt{\abs{t^2-\x^2}}(\x^2+2t^2)-12 \x \y^2 \z \cdot \arctan\left(\frac{\sqrt{\abs{ t^2-\x^2-\y^2} }}{\x}\right) 
  \end{gather}
  for $\sqrt{\x^2+\y^2}<t<\sqrt{\x^2+\y^2+\z^2}.$
\end{Theorem}

\begin{figure}[H]
  \label{fig:easy_3d}
  \centering
  \includegraphics[width=13cm]{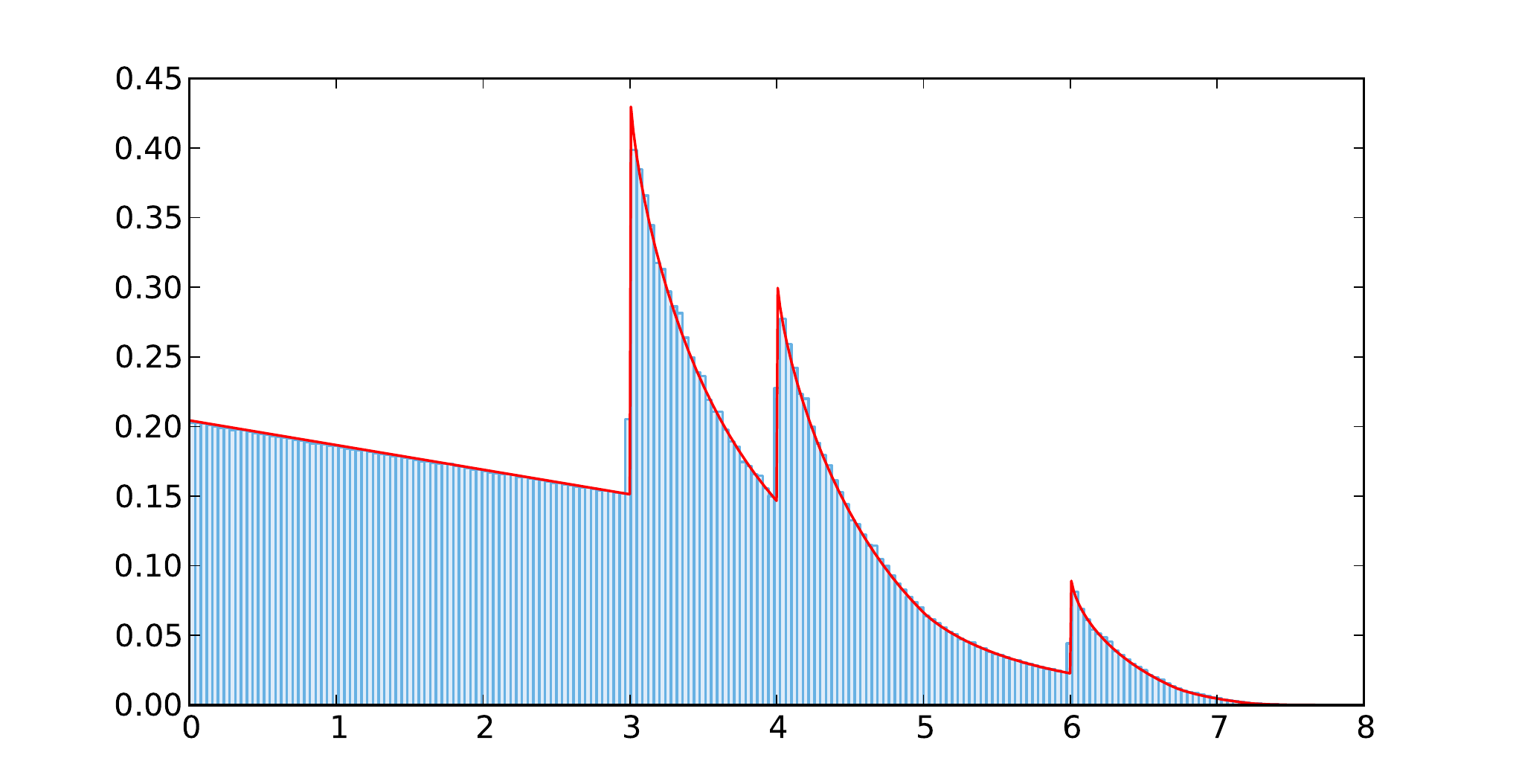}
  \caption{
    Simulation (blue histogram) vs explicit probability density
    function (red line) given by Theorem \ref{thm_3} for
    $(\x,\y,\z)=(3,4,6)$. (Simulation used $10^5$ particles, each
    starting at the origin with a uniformly random direction, going
    for a total distance $1000$ each.)  The fact that
    $\operatorname{pdf}_X(t)$ is not smooth at $t=5$ is barely
    noticeable. 
  }
\end{figure}

\begin{Remark}
  We note that the probability density function in Theorem \ref{thm_3} is analytic on all open subintervals of $(0, \sqrt{\x^2+\y^2+\z^2})$ not containing any of the points
  \begin{gather}
  \x,\y,\z,\sqrt{\x^2+\y^2}, \sqrt{\x^2+\z^2}, \sqrt{\y^2+\z^2}.
  \end{gather}
  Moreover, it is linear on the interval $(0,\min(\x,\y,\z))$ and has
  positive jump discontinuities at the points $\x,\y,\z$. At the
  points $\{\sqrt{\x^2+\y^2}, \sqrt{\x^2+\z^2},
  \sqrt{\y^2+\z^2}\}\setminus \{\x,\y,\z\}$, it is continuous and differentiable.
\end{Remark}

Note that the probability distribution $X_{M,R}$ gives a larger
``weight'' to some particles than others, since some particles get
more bounces than others for the same distance $R$. One could also
consider a similar problem where we send off each particle for a
certain number $N>0$ of bounces, and then consider the limit as
$M\to\infty$ followed by taking the limit $N\to\infty$, where $M$ is
the number of particles. This would give each particle the same
``weight''.  Denote the finite version of this distribution by
$Y_{M,N}$ and its limit distribution as $M\to\infty$ and then $N\to\infty$ by $Y$.  With
regard to the previous discussion about signal loss, we call the limit
distribution $X$ of $X_{M,R}$ the \dfn{spreading model} and we call
the limit distribution of $Y_{M,N}$ the \dfn{absorption
  model}. Determining the probability density function of the absorption model appears
to be the more difficult problem, and we give a formula only in dimension
two:

\begin{Theorem}
\label{thm_hard}
  For a box of dimension $n=2$ with side-lengths $\x\leq\y$, the random variable $Y_{M,N}$ converges in distribution
  to the random variable $Y$, as we take $M\to\infty$ followed by taking $N\to\infty$, where
  the probability density function $\op{pdf}_Y(t)$ is given by
    \begin{gather}
      \frac2\pi\Bigg(
      \frac{2 (\x+\y)}{ \left(\x^2+\y^2\right)}-\frac{2 \x \y }{\left(\x^2+\y^2\right)^{3/2}}\left(\tanh^{-1}\left(\frac{\x}{\sqrt{\x^2+\y^2}}\right) + \tanh^{-1}\left(\frac{\y}{\sqrt{\x^2+\y^2}}\right)\right)
      \Bigg)
    \end{gather}
    for $0<t<\x,\y$, and by
    \begin{gather}
      \frac{2}{\pi} \Bigg(
      \frac{\x \left(\y- \sqrt{t^{2}-\x^{2}}  \right) }
      {t  (\y+\sqrt{t^2-\x^2}) \sqrt{t^{2}-\x^{2}} } +
      \frac{2 \x \y +2\x t - 2\x\sqrt{t^2-\x^2}}{t \left(\x^2+\y^2\right)}+\\
      \frac{2 \x \y \left(-\tanh
       ^{-1}\left(\frac{t}{\sqrt{\x^2+\y^2}}\right)+\tanh ^{-1}\left(\frac{\sqrt{t^2-\x^2}
       \sqrt{\x^2+\y^2}}{t \y}\right) - \tanh^{-1}\left(\frac{\y}{\sqrt{\x^2+\y^2}}\right)\right)}{\left(\x^2+\y^2\right)^{3/2}}
       \Bigg)
    \end{gather}
    for $\x<t<\y$, and by
    \begin{gather}
      \frac{2}{\pi} \Bigg(
      \frac{\x(\y- \sqrt{t^{2}-\x^{2}})}{t  (\y+\sqrt{t^2-\x^2}) \sqrt{t^{2}-\x^{2}} } 
      +
      \frac{\y(\x- \sqrt{t^{2}-\y^{2}})}{t  (\x+\sqrt{t^2-\y^2}) \sqrt{t^{2}-\y^{2}} } 
      +
      2\frac{2\x \y- \x\sqrt{t^2-\x^2} - \y \sqrt{t^2-\y^2}}{t \left(\x^2+\y^2\right)}+ \\
       \frac{2 \x \y \left(-2\tanh
       ^{-1}\left(\frac{t}{\sqrt{\x^2+\y^2}}\right)+\tanh ^{-1}\left(\frac{\sqrt{t^2-\x^2}
       \sqrt{\x^2+\y^2}}{t \y}\right) +
        \tanh ^{-1}\left(\frac{\sqrt{t^2-\y^2}
           \sqrt{\x^2+\y^2}}{t \x}\right)
        \right)}{\left(\x^2+\y^2\right)^{3/2}}
      \Bigg)
    \end{gather}
    for $\x,\y<t<\sqrt{\x^2+\y^2}$.
\end{Theorem}

\begin{figure} \label{bounce_vs_distance}
  \centering
  \includegraphics[width=13cm]{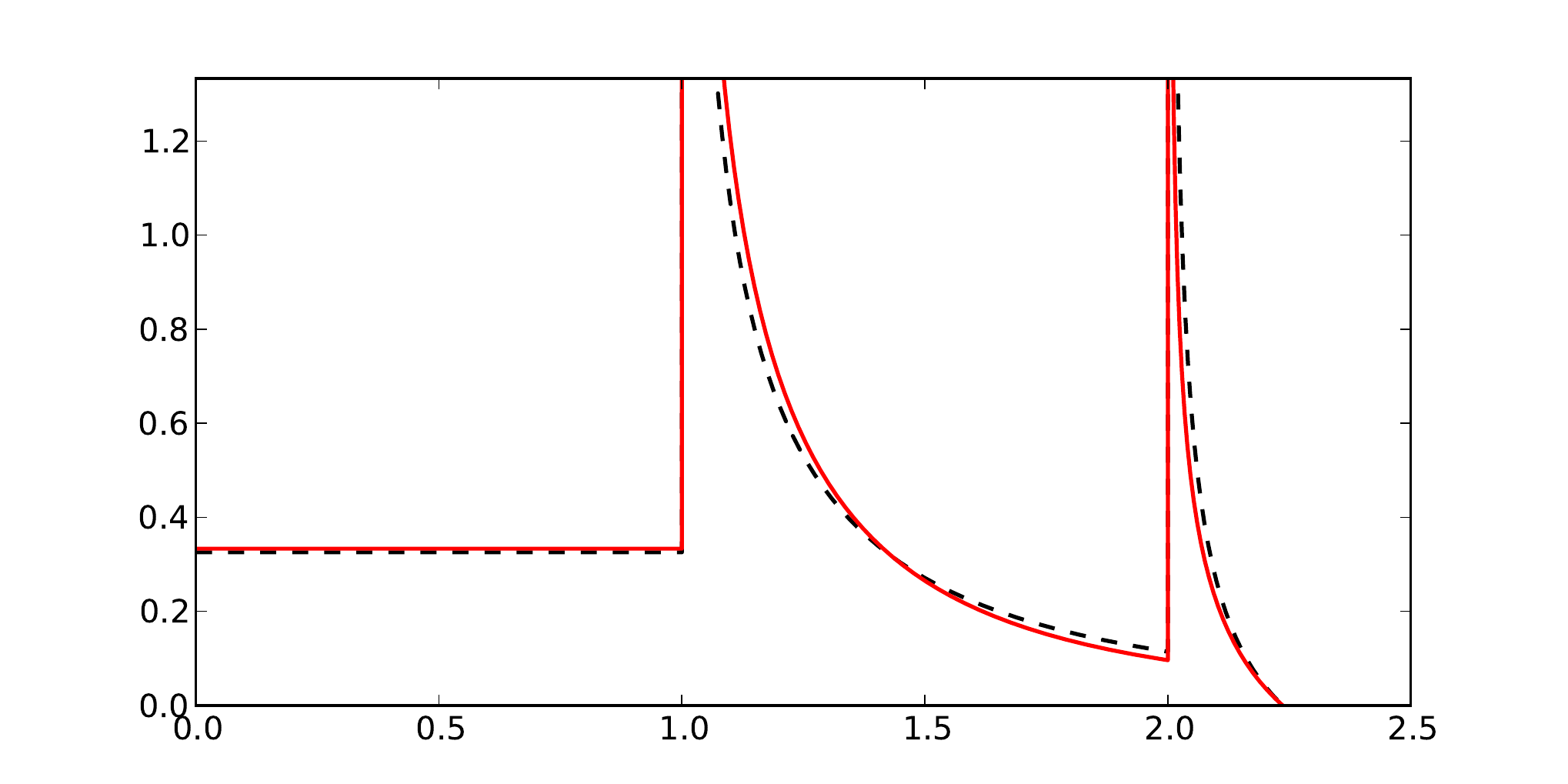}
  \caption{
    Probability density function for spreading model $X$ (red line) from Theorem \ref{thm_2} vs absorption model (black dashed line) from Theorem \ref{thm_hard}, for $(\x,\y)=(1,2)$.
  }
\end{figure}

See Figure \ref{bounce_vs_distance} for a comparison between the probability density functions for the two different models in dimension $2$.
\begin{Remark}
  It is not a priori obvious that the two limit distributions should
  differ, and it is natural to ask how much, if at all, they
  differ. We start by remarking that the expression for
  $\op{pdf}_Y(t)$ does not simplify into the expression for
  $\op{pdf}_X(t)$; indeed, for $(\x,\y)=(1,2)$ we have
  $\op{pdf}_X(t)=1/3$ but $\op{pdf}_Y(t)\approx 0.32553$ on the
  interval $(0,1)$.  For very skew boxes, with $\x=1$ and
  $\y\to\infty$, it is straightforward to show that
\begin{gather}
	\frac{\op{pdf}_Y(\y/2)} {\op{pdf}_X(\y/2)} \to \infty
\end{gather}
as $\y\to\infty$.
\end{Remark}

\subsection{Discussion}
\label{sec:discussion}

Given a closed convex subset $C \subset \R^{n}$ with nonempty interior
it is possible to define a natural
probability measure on the set of lines in $\R^{n}$ that have nonempty
intersection with $C$.  The expected length of the intersection of
a random line is then, up to a constant that only depends on $n$, 
given by $\operatorname{Vol}(C)/\operatorname{Area}(C)$; this is known
as Santalo's formula 
in the integral geometry 
and geometric probability literature (cf. \cite[Ch. 3]{santalo-book}).

A billiard flow on a manifold $M$ with boundary $\partial M$ gives
rise to a billiard map (roughly speaking, the phase space $\Omega$ is
then the collection of inward facing unit vectors $v$ at each point
$x \in \partial M$).  Given $(x,v) \in \Omega$ we define the
associated free path as the distance the billiard particle, starting
at $x$ in the direction $v$, covers before colliding with $\partial M$
again.  As the billiard map carries a natural probability measure
$\nu$ we can view the free path as a random variable, and the mean
free path is then just its expected value.  Remarkably, the mean free
path (again up to a constant that only depends on the dimension) is
then given by $\operatorname{Vol}(M)/\operatorname{Vol}(\partial M)$ --- even for non-convex
billiards.  This was deduced in the seventies at the Moscow seminar on
dynamical systems directed by Sinai and Alekseev but was never
published and hence rederived by a number of researchers.  For further
details and an interesting historical survey, see Chernov's paper
\cite[Sec.~2]{chernov-mean-free-path-billiards}.

In spirit our methods are closely related to the ones used by
Barra-Gaspard \cite{barra-gaspard-level-spacings-quantum-graphs} in
their study of the level spacing distribution for quantum 
graphs, and this turns out to be given by the distribution of return
times to a hypersurface of section of a linear flow on a torus.
In particular, for graphs with a finite number of disconnected bonds
of incommensurable lengths, the hypersurface of section is the
``walls'' of the torus, and the level spacings of the quantum graph is
exactly the same same as the free path length distribution in our
setting when all particles have the {\em same} starting velocity.  (In
particular, compare the numerator in \eqref{thm_n_formula} for $v$
fixed with \cite[Equation~(49)]{barra-gaspard-level-spacings-quantum-graphs}.)

In \cite{marklof-strombergsson-gaps-between-logs}, Marklof and
Strömbergsson used the results by Barra-Gaspard to determine the gap
distribution of the sequence of fractional parts of
$\{\log_{b} n \}_{n \in \Z^+}$.  The gap distribution depends on
whether $b$ is trancendental, rational or algebraic;  quite
remarkably the density function $P(s)$ for these gaps share a number
of qualitative features with the density function
$\operatorname{pdf}_X(s)$ for free paths in our setting.  Namely, the
density functions both have compact support and are smooth apart from
a finite number of jump discontinuities.  Further, in some cases the
density function is constant for $s$ small; compare
Figure~\ref{fig:easy-rectangle} (here $d=2$) with
\cite[Figure~4]{marklof-strombergsson-gaps-between-logs} (here
$b=\sqrt{10}$).  However, there are some important differences: for
$P(s)$, left and right limits exist at the jump discontinuities,
whereas for $d=2$, the right limit of $\operatorname{pdf}_X(s)$ is
$+\infty$ at the jumps (cf. Figure~\ref{fig:easy-rectangle}.)
Further, despite appearences, $P(s)$ is not 
linear near $s=0$
(cf. \cite[Figure~1]{marklof-strombergsson-gaps-between-logs}
corresponding to $b=e$) whereas for $d=3$, $\operatorname{pdf}_X(s)$
is indeed linear near $s=0$ (cf. Figure~\ref{fig:easy_3d}).

\subsection{Acknowledgements}
We would like to thank Z. Rudnick for some very helpful discussions,
especially for suggesting the connection with integral geometry.  We
also thank J. Marklof for bringing references
\cite{barra-gaspard-level-spacings-quantum-graphs,marklof-strombergsson-gaps-between-logs} to our attention.

{S.H. was partially supported by a grant from the
  Swedish Research Council (621-2011-5498).}
{P.K. was partially supported by grants from the G\"oran Gustafsson
  Foundation for Research in Natural Sciences and Medicine, and the
  Swedish Research Council (621-2011-5498).}

\section{Proof of Theorem \ref{thm_n}}
In this section, we prove Theorem \ref{thm_n}. 
For notational simplicity, we give the proof in dimension three; the general proof for $n\geq 2$ dimensions is analogous.

Given a particle with initial position $p$ and initial direction $v$, let $N_{R,p,v}$ be the number of bounce lengths we get from that particle as it has travelled a total distance $R>0$, and let $N_{R,p,v}(t)$ be the number of such bounce lengths of length at most $t\geq 0$. The uniform probability distribution on the set of bounce lengths of $M$ particles with initial positions $p^{(1)},\ldots,p^{(M)}$ and initial directions $v^{(1)},\ldots,v^{(M)}$ has the cumulative distribution function
\begin{gather}
\operatorname{cdf}_{X_{M,R} }(t)= 
\dfrac{\sum_{i=1}^M N_{R,p^{(i)},v^{(i)}}(t)}{\sum_{i=1}^M
  N_{R,p^{(i)},v^{(i)}}} 
= \dfrac{\frac1M\sum_{i=1}^M
  \dfrac{N_{R,p^{(i)},v^{(i)}}}R
  \dfrac{N_{R,p^{(i)},v^{(i)}}(t)}{N_{R,p^{(i)},v^{(i)}}}
}{\frac1M\sum_{i=1}^M  \dfrac{N_{R,p^{(i)},v^{(i)}}}R
}. \label{cdf_definition} 
\end{gather}
(Note that the denominator is uniformly bounded from below, which follows from equation \eqref{integrand_denominator_before_taking_limit} below.)
By the strong law of large numbers, the function \eqref{cdf_definition} converges almost
surely to
\begin{gather}
	\dfrac{\int_K\int_{\mathbb S^2} \dfrac{N_{R,p,v}}R
          \dfrac{N_{R,p,v}(t)}{N_{R,p,v}}\dd S(v)\dmu\mu
          p}{\int_K\int_{\mathbb S^2} \dfrac{N_{R,p,v}}R \dd
          S(v)\dmu\mu p } \label{cdf_integral} 
\end{gather}
as $M\to\infty$, where $d\mu$ is the probability measure with which we
choose the starting points, and $\dd S$ is the surface area measure on
the sphere $\mathbb S^2$.  By symmetry, we may restrict the inner
integrals to $\mathbb S^2_+\defeq \{(\vx,\vy,\vz)\in \mathbb S^2:
\vx,\vy,\vz> 0\}$. 
We now look at the limit of \eqref{cdf_integral} as $R\to\infty$, and
we note that since the integrands are uniformly bounded, we may move
the limit inside the integrals by the Lebesgue dominated convergence
theorem. Fix one of the integrands, and denote it by $f(R,p,v,t)$.
We will show that its limit $g(p,v,t)\defeq \lim_{R\to\infty}
f(R,p,v,t)$ exists for all $t$ and all directions $v\in \mathbb
S^2$. 
Moreover, if $p^{(i)}$ and $v^{(i)}$ denote random variables corresponding to
an initial position and an initial direction, respectively, as above, then
\begin{gather}
h(p^{(i)},v^{(i)},t) := \lim_{R \to \infty}
  \dfrac{N_{R,p^{(i)},v^{(i)}}}R
  \dfrac{N_{R,p^{(i)},v^{(i)} }(t)}{N_{R,p^{(i)},v^{(i)} }}
\end{gather}
is a random variable with
finite variance (and similarly for the terms in the denominator of
\eqref{cdf_definition}; in particular recall it is uniformly bounded
from below), 
and thus the strong law of large numbers gives that
the limit of \eqref{cdf_definition} as $R\to\infty$, and then $M \to\infty$
almost surely equals \eqref{cdf_integral}.   This shows that 
$\lim_{M\to\infty}\lim_{R\to\infty}\op{cdf}_{X_{M,R}}(t)$ exists
almost surely
and is equal to 
$\lim_{R\to\infty}\lim_{M\to\infty}\op{cdf}_{X_{M,R}}(t)$.

Consider a particle with initial position $p$ and initial direction
$v=(\vx,\vy,\vz)\in \mathbb S^2_+$. By ``unfolding'' its motion with
specular reflections on the walls of the box to the motion along a
straight line in $\R^n$ --- see Figure \ref{unfold_rgb} for a 2D
illustration --- we see that the particle's set of bounce lengths is
identical to the set of path lengths between consecutive intersections
of the straight line segment $\{p+tv: 0\leq t\leq R\}$ with any of the
planes $x=n\x, y=n\y, z=n\z$, $n\in\Z$. Thus we see that
\begin{gather} \label{integrand_denominator_before_taking_limit}
  N_{R,p,v} = R\frac{\vx}\x+R\frac{\vy}\y+R\frac{\vz}\z+\BigO(1)
\end{gather}
for
large $R$, and therefore
\begin{gather} \label{integrand_denominator}
	\dfrac{N_{R,p,v}}R\to \frac\vx\x+\frac\vy \y+\frac\vz\z
\end{gather}
as $R\to\infty$.

\begin{figure} \label{unfold_rgb}
  \centering
  \includegraphics[width=10cm]{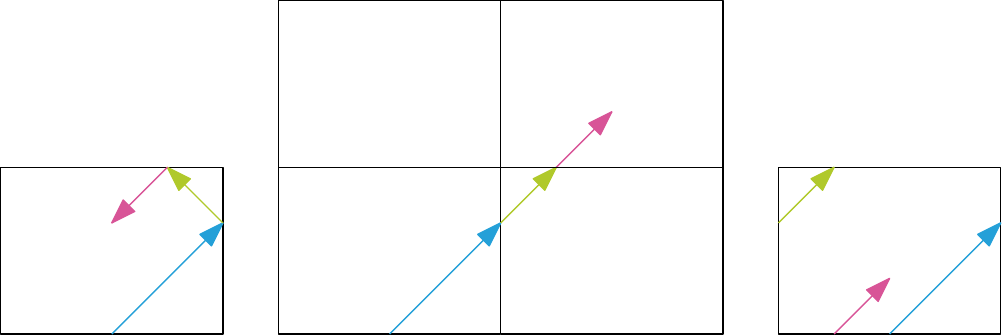}  
  \caption{From left to right: Unfolding a motion with specular reflection in a 2D box to a motion the plane and then projecting back to the box.
  }
\end{figure}

Now project the line $\{p+tv:0\leq t\leq R\}$ to the torus $\R^3/\Lambda$ where $\Lambda=\{(n_1\x,n_2\y,n_3\z): n_1,n_2,n_3\in\Z\}$ and let us identify the torus with the box $K$; see Figure \ref{unfold_rgb}. Each bounce length corresponds to a line segment which starts in one of the three planes $x=0$, $y=0$ or $z=0$ and runs in the direction $v$ to one of the three planes $x=\x, y=\y$ or $z=\z$. There are
$R\frac\vz\z+\BigO(1)$
line segments which start from the plane $z=0$, and thus the probability that a line segment starts from the plane $z=0$ is
\begin{gather}
	\dfrac{\frac\vz\z}{\frac\vx\x+\frac\vy\y+\frac\vz\z}
\end{gather}
as $R\to\infty$.  By the ergodicity of the linear flow on tori (for
almost all directions),
the starting
points of these line segments become uniformly distributed on the
rectangle $[0,\x]\times[0,\y]\times\{0\}$ for almost all
$v\in \mathbb S^2_+$ as $R\to\infty$; from here we will assume that
$v$ is such a direction, and we will ignore the measure zero set of
directions for which we do not have ergodicity. 
Consider one of these line segments and denote its length by $T$ and
its starting point by $(x_0,y_0,0)$. For an arbitrary parameter
$t\geq 0$, we have $T\leq t$ if and only if $t\vx\geq \x-x_0$ or
$t\vy\geq \y-y_0$ or $t\vx\geq \z$; the starting points
$(x_0,y_0)\in[0,\x]\times[0,\y]$ which satisfy this are precisely
those outside the rectangle $[0,\x-t\vx]\times[0,\y-t\vy]$ assuming
that $t\vz\leq \z$ and otherwise it is the whole rectangle
$[0,\x]\times[0,\y]$. The area of that region is
\begin{gather}
  \x\y-(\x-t\vx)(\y-t\vy) \label{prevcalc}
\end{gather}
if $\x\geq t\vx, \y\geq t\vy, \z\geq t\vz$ and otherwise it is $\x\y$. Since the starting points $(x_0,y_0)$ are uniformly distributed in the rectangle $[0,\x]\times[0,\y]$ as $R\to\infty$, it follows that the probability that $T\leq t$ is
\begin{gather}
	1-\dfrac{(\x-t\vx)(\y-t\vy)}{\x\y}\chi(\x\geq t\vx, \y\geq t\vy, \z\geq t\vz),
\end{gather}
where $\chi(P)$ is the indicator function which is $1$ whenever the condition $P$ is true, and $0$ otherwise.
We get analogous expressions for the case when a line segment starts
in the plane $x=0$ or $y=0$ instead. Thus the
proportion
   of all line 
segments with length at most $t$ as $R\to\infty$ is 
\begin{small}
  \begin{align}
  	\lim_{R\to\infty} \dfrac{N_{R,p,v}(t)}{N_{R,p,v}} &=
  	\dfrac{\frac\vx\x}{\frac\vx\x+\frac\vy\y+\frac\vz\z} \left(1-\dfrac{(\y-t\vy)(\z-t\vz)}{\y\z}\chi(\x\geq t\vx, \y\geq t\vy, \z\geq t\vz)\right)+\\
  	&\phantom{=.}\dfrac{\frac\vy\y}{\frac\vx\x+\frac\vy\y+\frac\vz\z} \left(1-\dfrac{(\x-t\vx)(\z-t\vz)}{\x\z}\chi(\x\geq t\vx, \y\geq t\vy, \z\geq t\vz)\right)+\\
  	&\phantom{=.}\dfrac{\frac\vz\z}{\frac\vx\x+\frac\vy\y+\frac\vz\z} \left(1-\dfrac{(\x-t\vx)(\y-t\vy)}{\x\y}\chi(\x\geq t\vx, \y\geq t\vy, \z\geq t\vz)\right) \label{f_from_convergence_in_distribution}
  \end{align}
\end{small}
which can be written
\begin{multline}
\label{integrand_numerator}
  	1-\dfrac{\chi(\x\geq t\vx, \y\geq t\vy, \z\geq t\vz)}{ \x\y\z
          (\frac\vx\x+\frac\vy\y+\frac\vz\z)} \times \\ \times
    \bigg(
  		\vx(\y-t\vy)(\z-t\vz)+
  		\vy(\x-t\vx)(\z-t\vz)+
  		\vz(\x-t\vx)(\y-t\vy)
  	\bigg).
\end{multline}
Recognizing that both integrands \eqref{integrand_denominator} and \eqref{integrand_numerator} are independent of the position $p$, we see that the limit of \eqref{cdf_integral} as $R\to\infty$ may be written as
\newcommand{\vint}[3]{ \int_{\substack{
			v\in \mathbb S^2_+\\
			\vx\leq #1/t \\
			\vy\leq #2/t \\
			\vz\leq #3/t   }} }
\begin{multline}
\label{cdf_formula}
  \lim_{R\to\infty}\lim_{M\to\infty} \op{cdf}_{X_{M,R}}(t) =
	1-
	\dfrac1
	{ \int_{\mathbb S^2_+}
		(\vx\y\z+\x\vy\z+\x\y\vz)
	\dd S(v) }\times\\
	\times\int_{\substack{
			v\in \mathbb S^2_+\\
			\vx\leq \x/t \\
			\vy\leq \y/t \\
			\vz\leq \z/t   }}
		(
		(\x\y\vz+\x\vy\z+\vx\y\z)
		-2t(\x\vy\vz+\vx\y\vz+\vx\vy\z)
		+3t^2\vx\vy\vz
	)
	\dd S(v)
\end{multline}
for all $t>0$.
The corresponding formula in $n$ dimensions is given by
  \begin{gather} \label{thm_n_formula}
    \lim_{R\to\infty}\lim_{M\to\infty} \op{cdf}_{X_{M,R}}(t) =
    1-
    \dfrac{
      \displaystyle \int_{\substack{
      v\in \mathbb S^{n-1}_+\\
      v_i \leq \x_i/t \\ \text{ for }i=1,\ldots,n   }}
        \left(
          \sum_{i=1}^n
            v_i \prod_{j\neq i}(a_i-tv_j)
        \right)
      \dd S(v)
    }
    {\displaystyle \left(\prod_{i=1}^n \x_i\right) \int_{\mathbb S^{n-1}_+}
      \left( \sum_{i=1}^n \frac{v_i}{\x_i} \right)
    \dd S(v) }
  \end{gather}
for all $t>0$, where the side-lengths of the box $K$ are
$\x_1,\ldots,\x_n$ and $\dd S$ is the surface area measure on $\mathbb
S^{n-1}_+\cap [0,\infty)^n$. (The denominator can be given explicitly by
using Lemma \ref{integral_of_vn} below.) 

We have thus proved that the random variable $X_{M,R}$ converges
in distribution to a random variable with probability density function
given by \eqref{thm_n_formula} as we take $M\to\infty$ followed by
taking $R\to\infty$, or alternatively, first taking $R\to\infty$ followed by taking $M\to\infty$.
It remains to prove that this distribution agrees with
the distribution of the random variable $X$ defined in the
introduction. 

\subsection{Integral geometry}
We start by recalling some standard facts from integral geometry (cf.
\cite{santalo-book,klain-rota-integral-geometry-book}.)
The set of directed straight lines $\ell$ in $\R^3$ can be
parametrized by pairs $(v,q)$ where $v\in \mathbb S^2$ is a unit
vector pointing in the same direction as $\ell$ and $q\in v^\bot$ is
the unique point in $\ell$ which intersects the plane through the
origin which is orthogonal to $v$. The unique translation- and
rotation-invariant measure (up to a constant) on the set of directed
straight lines in $\R^3$ is $\dd \ell\defeq \dd A(q) \dd S(v)$ where
$\dd A$ is the surface measure on the plane through the origin
orthogonal to $v\in \mathbb S^2$, and $\dd S$ is the surface area
measure on $\mathbb S^2$. 

Consider the set $L_{\x,\y,\z}$ of directed straight lines in $\R^3$
which intersect the box $K$. Now, since $\x\y\vz+\x\vy\z+\vx\y\z$ is
the area of the projection 
of the box $K$ onto the plane $v^\bot$ for
$v\in \mathbb S^2_+$, it follows that the total measure of $L_{a,b,c}$ with
respect to $\dd \ell$ is
\begin{gather}
	C_{\x,\y,\z} \defeq 8\int_{\mathbb
          S^2_+}(\x\y\vz+\x\vy\z+\vx\y\z)\dd S(v) = 2\pi
        (\x\y+\x\z+\y\z) 
\end{gather}
where we used symmetry, and the integral may be evaluated by switching to spherical coordinates. It follows that $\dd \ell / C_{\x,\y,\z}$ is a probability measure on the set of directed lines intersecting the box $L_{a,b,c}$. Let $\ell$ be a random directed line with respect to this measure, and define the random variable $X\defeq \operatorname{length}(\ell\cap K)$, as in the introduction. Let us determine the probability that $X\leq t$ for an arbitrary parameter $t\geq 0$. By symmetry it suffices to consider only directed lines with $v\in \mathbb S^2_+$.
The set of all intersection points between the rectangle
$[0,\x]\times[0,\y]\times\{0\}$ and the lines $\ell$ with $X\leq t$
and direction $v\in \mathbb S^2_+$ has area
$\x\y-(\x-t\vx)(\y-t\vy)\chi(\x\geq t\vx,\y\geq t\vy,\z\geq t\vz)$, as
in \eqref{prevcalc}, and its projection onto the plane $v^\bot$ has
area 
\begin{gather}
	\vz\left[\x\y-(\x-t\vx)(\y-t\vy)\chi(\x\geq t\vx,\y\geq t\vy,\z\geq t\vz)\right].
\end{gather}
By symmetry it follows that the area of the set of directed lines
$\ell\in L_{\x,\y,\z}$ with $X\leq t$ and direction $v\in \mathbb
S^2_+$ projected down to $v^\bot$ is 
\begin{align}
	U(v,t) &\defeq \vx\left[\y\z-(\y-t\vy)(\z-t\vz)\chi(\x\geq t\vx,\y\geq t\vy,\z\geq t\vz)\right] + \\
	&\phantom{\defeq.}\vy\left[\x\z-(\x-t\vx)(\z-t\vz)\chi(\x\geq t\vx,\y\geq t\vy,\z\geq t\vz)\right] + \\
	&\phantom{\defeq.} \vz\left[\x\y-(\x-t\vx)(\y-t\vy)\chi(\x\geq t\vx,\y\geq t\vy,\z\geq t\vz)\right],
\end{align}
and it follows that
\begin{gather}
	\operatorname{Prob}[X\leq t] = \dfrac1{C_{\x,\y,\z}} \int_{X\leq t}\dd \ell =
		\dfrac 8{C_{\x,\y,\z}}\int_{\mathbb S^2_+}
		U(v,t)
		\dd S(v),
\end{gather}
which we see is identical to \eqref{cdf_formula}, and we have thus proved that $X_{M,R}$ converges in distribution to $X$ as we take $M\to\infty$ and then $R\to\infty$. This concludes the proof of Theorem \ref{thm_n}. 

\pp {Computing the mean value} 
\label{sec:integral-geometry-mean-value}
We will determine the mean value \eqref{santalo-formula} of $X$; to do this we
exploit the integral geometry interpretation of the random
variable $X$.  By symmetry it suffices to restrict to directed lines
$\ell$ with $v\in \mathbb S^2_+$.  For fixed $v\in \mathbb S^2_+$,
denote by $Q(v)=(K+\operatorname{span}(v))\cap v^\bot$ the set of
$q\in v^\bot$ such that the directed line $\ell$ parametrized by
$(v,q)$ intersects $K$.  We note that $X\dd A(q)$ is a volume element
of the box $K$ for any fixed $v\in \mathbb S^2_+$, and thus
integrating $X\dd A(q)$ over all $q$ yields the volume of the box.
Hence the mean value is 
\begin{gather}
	\E{X}=
	\dfrac8{C_{\x,\y,\z}}
	\int_{\mathbb S^2_+}\int_{Q(v)}X\dd A(q)\dd S(v) =
	\dfrac{8\x\y\z} {C_{\x,\y,\z}}
	\int_{\mathbb S^2_+} \dd S(v) =
	\dfrac{2\x\y\z}{ \x\y+\x\z+\y\z}.
\end{gather}
In $n$ dimensions we get a normalizing factor
$\frac{\operatorname{Area}(K)}2 \cdot 2^n\int_{\mathbb S^{n-1}_+} v_n\dd
S(v)$, so with the aid of the Lemma \ref{integral_of_vn} in the
Appendix, it follows that the mean value in $n$ dimensions is 
\begin{gather}
  \E{X} =
  \dfrac1{2^n   \dfrac1{\pi}\dfrac{|\mathbb S^n|}{2^{n}}
    \frac{\operatorname{Area}(K)}2  }
  2^n\operatorname{Vol}(K)\dfrac{|\mathbb S^{n-1}|}{2^n} = 2\pi
  \dfrac{|\mathbb S^{n-1}|}{|\mathbb
    S^n|}\dfrac{\operatorname{Vol}(K)}{\operatorname{Area}(K)}
\end{gather}
where $\operatorname{Area}(K)$ is the $(n-1)$-dimensional surface area
of the box $K$, and $\operatorname{Vol}(K)$ is the volume of the box
$K$. 

\section{Proof of Theorem \ref{thm_2}}
Using formula \eqref{thm_n_formula} in dimension $n=2$, we
get
\begin{gather}
      \label{eq:pdf-singularity}
    \op{cdf}_X(t) =
    1-
    \dfrac{
      \displaystyle \int_{\substack{
      v\in \mathbb S^1_+\\
      \vx \leq \x/t \\
      \vy \leq \y/t}}
        \left(
          \vx (\y-t\vy) +
          \vy (\x-t\vx)
        \right)
      \dd S(v)
    }
    {\displaystyle \x\y \int_{\mathbb S^1_+}
      \left( \frac{\vx}\x+\frac{\vy}\y \right)
      \dd S(v) }.
\end{gather}

We use polar coordinates $\vx=\cos\theta, \vy=\sin\theta$ so that $\dd
S(v)=\dd\theta$. Then the above becomes 
\begin{gather}
    1-
    \dfrac{
      \displaystyle
      \int_{ \cos\inv (\min(a/t, 1)) } ^ { \sin\inv (\min(b/t, 1)) }
        \left(
          \y\cos\theta+\x\sin\theta-2t\sin\theta\cos\theta
        \right)
      \dd \theta
    }
    {\displaystyle \int_{0}^{\pi/2}
      \left( \y\cos\theta + \x\sin\theta \right)
    \dd \theta } = \\
    1-
    \frac1{\x+\y}
      \left[
        \y\sin\theta-\x\cos\theta+t\cos^2\theta
      \right]
      _{ \cos\inv (\min(a/t, 1)) } ^ { \sin\inv (\min(b/t, 1)) }. \label{two_easy}
    \end{gather}
The numerator of the second term may be written
\begin{gather}
  \chi(\y< t)\left( \y\cdot \frac \y t-\x\sqrt{1-\frac{\y^2}{t^2}}+t\left(1-\frac{\y^2}{t^2}\right) \right) + 
  \chi(\y\geq t)\left( \y - \x\cdot 0 + t\cdot 0  \right) + \\
  -\chi(\x< t)\left( \y \sqrt{1-\frac{\x^2}{t^2}} - \x\cdot\frac{\x}t + t\cdot\frac{\x^2}{t^2} \right)
  -\chi(\x \geq t)\left( \y\cdot 0 - \x+t \right)
 \end{gather}
which can be simplified to
 \begin{gather}
  \chi(\y< t)\left( t-\y-\x\sqrt{1-\frac{\y^2}{t^2}} \right) + 
  \chi(\x< t)\left( t - \x-\y \sqrt{1-\frac{\x^2}{t^2}}  \right) + 
  \left(  \x+\y-t \right).
\end{gather}
Inserting this into \eqref{two_easy} and differentiating yields Theorem \ref{thm_2}.

\section{Proof of Theorem \ref{thm_3}}
We will evaluate the cumulative distribution function \eqref{cdf_formula} and then differentiate.
The denominator of the second term of \eqref{cdf_formula} is
\begin{gather}
	\int_{\mathbb S^2_+}(\x\y\vz+\x\vy\z+\vx\y\z)\dd S(v) =
        \dfrac{\pi}4 (\x\y+\x\z+\y\z), 
\end{gather}
as may be evaluated by switching to spherical coordinates.
Define
\begin{gather}
	f(\x,\y,\z) \defeq \y\z \vint\x\y\z \vx \dd S(v),\\
	g(\x,\y,\z) \defeq -2t\z \vint\x\y\z \vx\vy \dd S(v),\\
	h(\x,\y,\z) \defeq 3t^2 \vint\x\y\z \vx\vy\vz \dd S(v).
\end{gather}
By symmetry, we have
\begin{gather}
	f(\z,\x,\y) =
	\x\y \vint\z\x\y \vx \dd S(v) =
	\x\y \vint\x\y\z \vz \dd S(v),\\
	f(\y,\z,\x) =
	\x\z \vint\y\z\x \vx \dd S(v) =
	\x\z \vint\x\y\z \vy \dd S(v),\\
	g(\z,\x,\y) =
	-2t\y \vint\z\x\y \vx\vy \dd S(v) =
	-2t\y \vint\x\y\z \vx\vz \dd S(v), \\
	g(\y,\z,\x) =
	-2t\x \vint\y\z\x \vx\vy \dd S(v) =
	-2t\x \vint\x\y\z \vy\vz \dd S(v),
\end{gather}
and thus we can write the numerator in the second term of \eqref{cdf_formula} as
\begin{gather}
	f(\x,\y,\z)+
	f(\z,\x,\y)+
	f(\y,\z,\x)+
	g(\x,\y,\z)+
	g(\z,\x,\y)+
	g(\y,\z,\x)+
	h(\x,\y,\z).
\end{gather}
Exploiting the symmetries, it suffices to evaluate $h(\x,\y,\z), g(\x,\y,\z)$ and $f(\y,\z,\x)$ (note the order of the arguments to $f$). 
We will evaluate these integrals by switching to spherical coordinates, but first we need to parametrize the part of the sphere inside the box $0\leq \vx\leq \x/t, 0\leq \vy\leq \y/t, 0\leq \vz\leq \z/t$.

\newcommand{\thetamin}{{\theta_{\text{min}}}}
\newcommand{\thetamax}{{\theta_{\text{max}}}}
\begin{Lemma} \label{param_lemma}
  Fix $t\in(0,\sqrt{\x^2+\y^2+\z^2})$. We have
  \begin{gather}
    \vint\x\y\z F(\vx,\vy,\vz) \dd S(v) = \\
      \left(
        \int_\thetamin^{\theta_\x}\int_0^{\pi/2} +
        \int_{\theta_\x}^{\thetamax} \int_{\varphi_\x}^{\pi/2} -
        \int_{\theta_\y}^\thetamax\int_{\varphi_\y}^{\pi/2}
      \right) \tilde F(\theta,\varphi) \sin\theta \dd\varphi\dd\theta
  \end{gather}
  for any integrable function $F: \mathbb S^2_+\to\R$, where $\tilde
  F(\theta,\varphi) \defeq F(\sin\theta\cos\varphi,
  \sin\theta\sin\varphi, \cos\theta)$, where 
  \begin{align}
    \thetamin &\defeq \acos\fp{\frac \z t}, \\
    \theta_\x &\defeq \max(\thetamin, \asin\fp{\frac \x t}), \\
    \theta_\y &\defeq \max(\thetamin, \asin\fp{\frac \y t}), \\
    \thetamax &\defeq \asin\fp{\dfrac{\sqrt{\x^2+\y^2}}t}, \\
    \varphi_\x &\defeq \acos{\frac \x{t\sin\theta}}\qquad(\text{whenever }\x\leq t\sin\theta),\\
    \varphi_\y &\defeq \asin{\frac \y{t\sin\theta}}\qquad(\text{whenever }\y\leq t\sin\theta).
  \end{align}
  and where we have used the shorthand $\fp{u}\defeq\min(u,1)$.
\end{Lemma}

\begin{proof}
We will parametrize the set of points $v=(\vx,\vy,\vz)$ on the sphere
$\mathbb S^2$ such that 
\begin{gather}
	0< \vx\leq \x/t,\\
	0< \vy\leq \y/t, \label{conds} \\
	0< \vz\leq \z/t.
\end{gather}
Switch to spherical coordinates $\vx=\sin\theta\cos\varphi, \vy=\sin\theta\sin\varphi, \vz=\cos\theta$. The non-negativity conditions of \eqref{conds} are equivalent to the condition $\theta,\varphi\in(0,\pi/2)$. For such angles, the condition $\vz\leq \z/t$ is equivalent to
\begin{gather}
	\acos\fp{\frac \z t} \leq \theta,
\end{gather}
and the conditions $\vx\leq \x/t, \vy\leq \y/t$ are equivalent to
\begin{gather}
	\acos\fp{\dfrac \x{t\sin\theta}} \leq \varphi \leq \asin\fp{\dfrac \y{t\sin\theta}}. \label{theta_interval}
\end{gather}
The interval \eqref{theta_interval} is non-empty for precisely those $\theta\in(0,\pi/2)$ such that $\theta\leq \thetamax$ since
\begin{gather}
  1 \leq \fp{\dfrac \x{t\sin\theta}}^2 + \fp{\dfrac \y{t\sin\theta}}^2 \iff
  1 \leq \left({\dfrac \x{t\sin\theta}}\right)^2 + \left({\dfrac \y{t\sin\theta}}\right)^2 \iff \\
  \sin\theta \leq \dfrac{\sqrt{\x^2+\y^2}}t  \iff \theta \leq \asin\fp{\dfrac{\sqrt{\x^2+\y^2}}t}.
\end{gather}
Thus we may restrict $\theta$ to the interval given by the inequalities
\begin{gather}
  \thetamin
  \leq  \theta \leq 
  \thetamax.
\end{gather}
Note that we have $\thetamin \leq \thetamax$ for all $t\leq\sqrt{\x^2+\y^2+\z^2}$ since
\begin{gather}
  \thetamin \leq \thetamax \iff
  1 \leq \fp{\frac \z t}^2 + \fp{\frac{\sqrt{\x^2+\y^2}}t}^2 \iff \\
  1 \leq \left({\frac \z t}\right)^2 + \left({\frac{\sqrt{\x^2+\y^2}}t}\right)^2 \iff
  t^2 \leq \x^2+\y^2+\z^2.
\end{gather}
We conclude that we can write
\begin{gather}
	\vint\x\y\z F(\vx,\vy,\vz) \dd S(v) =
	\int_{\thetamin}^{\thetamax} \int_{\acos\fp{\frac \x{t\sin\theta}}}^{ \asin\fp{\frac \y{t\sin\theta}} }
		\tilde F(\theta,\varphi) \sin\theta \dd\varphi\dd\theta. \label{integral_reparametrization}
\end{gather}

For $\theta\in(0,\pi/2)$, note that $\acos{\frac \x{t\sin\theta}}$ is defined  precisely when
$\asin\fp{\frac \x t}\leq \theta$
and that $\asin{\frac \y{t\sin\theta}}$ is defined precisely when
$\asin\fp{\frac \y t}\leq \theta$.
  We have $\thetamin< \theta_\x$ if and only if $t< \sqrt{\x^2+\z^2}$, and we have $\thetamin<\theta_\y$ if and only if $t<\sqrt{\y^2+\z^2}$. Moreover we note that we always have $\theta_\x, \theta_\y \in [\thetamin, \thetamax]$.

Let us rewrite the integration limits in the right-hand side of \eqref{integral_reparametrization} in terms of $\varphi_\x$ and $\varphi_\y$.
A priori, we need to distinguish between the two cases $\theta_\x\leq \theta_\y$ and $\theta_\y< \theta_\x$. If $\theta_\x\leq \theta_\y$ then we get
\begin{gather}
  \left(\int_{\thetamin}^{\thetamax} \int_{\acos\fp{\frac x{t\sin\theta}}}^{ \asin\fp{\frac y{t\sin\theta}} }\right) =
	\left(
		\int_\thetamin^{\theta_\x}\int_0^{\pi/2} +
		\int_{\theta_\x}^{\theta_\y} \int_{\varphi_\x}^{\pi/2} +
		\int_{\theta_\y}^\thetamax\int_{\varphi_\x}^{\varphi_\y}
	\right) = \\ 
	\left(
		\int_\thetamin^{\theta_\x}\int_0^{\pi/2} +
		\int_{\theta_\x}^{\thetamax} \int_{\varphi_\x}^{\pi/2} -
		\int_{\theta_\y}^{\thetamax} \int_{\varphi_\x}^{\pi/2} +
		\int_{\theta_\y}^\thetamax\int_{\varphi_\x}^{\pi/2} -
		\int_{\theta_\y}^\thetamax\int_{\varphi_\y}^{\pi/2}
	\right) = \\ 
	\left(
		\int_\thetamin^{\theta_\x}\int_0^{\pi/2} +
		\int_{\theta_\x}^{\thetamax} \int_{\varphi_\x}^{\pi/2} -
		\int_{\theta_\y}^\thetamax\int_{\varphi_\y}^{\pi/2}
	\right). 
 \label{a_priori_first_case}
\end{gather}
If on the other hand $\theta_\y<\theta_\x$ then 
\begin{gather}
  \left(\int_{\thetamin}^{\thetamax} \int_{\acos\fp{\frac x{t\sin\theta}}}^{ \asin\fp{\frac y{t\sin\theta}} }\right) =
	\left(
		\int_\thetamin^{\theta_\y}\int_0^{\pi/2} +
		\int_{\theta_\y}^{\theta_\x} \int_0^{\varphi_\y} +
		\int_{\theta_\x}^\thetamax\int_{\varphi_\x}^{\varphi_\y}
	\right) = \\ 
	\left(
		\int_\thetamin^{\theta_\y}\int_0^{\pi/2} +
		\int_{\theta_\y}^{\theta_\x} \int_0^{\pi/2} -
		\int_{\theta_\y}^{\theta_\x} \int_{\varphi_\y}^{\pi/2} +
		\int_{\theta_\x}^\thetamax\int_{\varphi_\x}^{\pi/2} -
		\int_{\theta_\x}^\thetamax\int_{\varphi_\y}^{\pi/2}
	\right)
\end{gather}
which we see is identical to \eqref{a_priori_first_case}. Combining \eqref{integral_reparametrization} and \eqref{a_priori_first_case} we get the conclusion of the lemma.
\end{proof}

Applying Lemma \ref{param_lemma} 
we get
\begin{gather}
	h(\x,\y,\z) = 3t^2 \vint\x\y\z \vx\vy\vz \dd S(v) = \\
	3t^2
	\left(
		\int_\thetamin^{\theta_\x}\int_0^{\pi/2} +
		\int_{\theta_\x}^{\thetamax} \int_{ \varphi_\x }^{\pi/2} -
		\int_{\theta_\y}^\thetamax\int_{ \varphi_\y }^{\pi/2}
	\right)
		(\sin^2\theta\cos\theta\cos\varphi\sin\varphi) \sin\theta
	\dd\varphi\dd\theta.
\end{gather}
An antiderivative of the integrand $\cos\varphi\sin\varphi \cdot \sin^3\theta\cos\theta$ with respect to $\varphi$ is
$-\frac{1}{2} \cos ^2\varphi \sin ^3\theta \cos \theta$, and thus the above is
\begin{gather}
	3t^2
	\left(
		\int_\thetamin^{\theta_\x}
			\left. \cos ^2\varphi \right|_{\varphi=0} +
		\int_{\theta_\x}^{\thetamax}
			\left. \cos ^2\varphi \right|_{\varphi=\varphi_\x} -
		\int_{\theta_\y}^\thetamax
			\left. \cos ^2\varphi \right|_{\varphi=\varphi_\y}
	\right)
		\frac12 \sin ^3\theta \cos \theta
	\dd\theta = \\
	3t^2
	\left(
		\int_\thetamin^{\theta_\x}
			1 +
		\int_{\theta_\x}^{\thetamax}
			\dfrac{a^2}{t^2\sin^2\theta} +
		\int_{\theta_\y}^\thetamax
			\left(\dfrac{b^2}{t^2\sin^2\theta}-1\right)
	\right)
		\frac12 \sin ^3\theta \cos \theta
	\dd\theta = \\
	\frac32
	\left(
		\int_\thetamin^{\theta_\x}
			t^2\sin ^3\theta \cos \theta \dd\theta +
		\int_{\theta_\x}^{\thetamax}
			a^2 \sin \theta \cos \theta \dd\theta +
		\int_{\theta_\y}^\thetamax
			\left(b^2\sin\theta - t^2\sin^3\theta \right) \cos \theta \dd\theta
	\right) = \\
	\frac32
	\left(
		\left[
			t^2\frac14\sin ^4\theta
		\right]_\thetamin^{\theta_\x} +
		\left[
			a^2 \frac12 \sin^2 \theta 
		\right]_{\theta_\x}^{\thetamax}+
		\left[
			b^2\frac12\sin^2\theta-t^2\frac14\sin^4\theta
		\right]_{\theta_\y}^\thetamax
	\right). \label{h_explicit}
\end{gather}

Next consider
\begin{gather}
  g(\x,\y,\z) = -2t\z \vint\x\y\z \vx\vy \dd S(v) = \\
  -2t\z
  \left(
    \int_\thetamin^{\theta_\x}\int_0^{\pi/2} +
    \int_{\theta_\x}^{\thetamax} \int_{ \varphi_\x }^{\pi/2} -
    \int_{\theta_\y}^\thetamax\int_{ \varphi_\y }^{\pi/2}
  \right)
    (\sin^2\theta\cos\varphi\sin\varphi) \sin\theta
  \dd\varphi\dd\theta.
\end{gather}
An antiderivative of the integrand $\cos\varphi\sin\varphi \cdot \sin^3\theta$ with respect to $\varphi$ is
$-\frac{1}{2} \cos ^2\varphi \sin ^3\theta$, and thus the above is
\begin{gather}
  g(\x,\y,\z) = -2t\z \vint\x\y\z \vx\vy \dd S(v) = \\
  -t\z
  \left(
    \int_\thetamin^{\theta_\x} \left.\cos ^2\varphi\right|_{\varphi=0} +
    \int_{\theta_\x}^{\thetamax} \left.\cos ^2\varphi\right|_{\varphi=\varphi_\x} -
    \int_{\theta_\y}^\thetamax \left.\cos ^2\varphi\right|_{\varphi= \varphi_\y }
  \right)
    \sin ^3\theta
  \dd\theta = \\
  -t\z
  \left(
    \int_\thetamin^{\theta_\x} 1 +
    \int_{\theta_\x}^{\thetamax} \dfrac{a^2}{t^2\sin^2\theta} +
    \int_{\theta_\y}^\thetamax \left(\dfrac{b^2}{t^2\sin^2\theta}-1\right)
  \right)
    \sin ^3\theta
  \dd\theta = \\
  -t\z
  \left(
    \int_\thetamin^{\theta_\x} \sin ^3\theta\dd\theta +
    \int_{\theta_\x}^{\thetamax} \dfrac{a^2\sin\theta}{t^2}\dd\theta +
    \int_{\theta_\y}^\thetamax \left(\dfrac{b^2\sin \theta}{t^2}-\sin ^3\theta\right)\dd\theta
  \right) = \\
  -t\z
  \left(
    \left[\frac{\cos^3\theta}3-\cos\theta\right]_\thetamin^{\theta_\x}  +
    \dfrac{a^2}{t^2}\left[-\cos\theta\right]_{\theta_\x}^{\thetamax}  +
    \left[-\dfrac{b^2\cos \theta}{t^2}-\frac{\cos^3\theta}3+\cos\theta\right]_{\theta_\y}^\thetamax
  \right). \label{g_explicit}
\end{gather}
We obtain $g(\y,\z,\x)$ and $g(\z,\x,\y)$ by switching the roles of
$\x,\y,\z$ in \eqref{g_explicit}. We remark that trying to obtain
$g(\y,\z,\x)$ and $g(\z,\x,\y)$ directly, by integrating $\vy\vz$ and
$\vx\vz$, respectively, by first integrating with respect to
$\varphi$, taking the limits $\varphi\to\varphi_a$ and
$\varphi\to\varphi_b$, and then finding an antiderivative with respect
to $\theta$, seem to result in much more complicated expressions.

Finally consider
\begin{gather}
	f(\y,\z,\x) =
	\x\z \vint\x\y\z \vy \dd S(v) = \\
	\x\z
	\left(
		\int_\thetamin^{\theta_\x}\int_0^{\pi/2} +
		\int_{\theta_\x}^{\thetamax} \int_{ \varphi_\x }^{\pi/2} -
		\int_{\theta_\y}^\thetamax\int_{ \varphi_\y }^{\pi/2}
	\right)
		(\sin\theta\sin\varphi) \sin\theta
	\dd\varphi\dd\theta.
\end{gather}
An antiderivative of the integrand $\sin\varphi\cdot \sin^2\theta$ with respect to $\varphi$ is $-\cos\varphi\cdot \sin^2\theta$, and thus the above is
\begin{gather}
	\x\z	
	\left(
		\int_\thetamin^{\theta_\x}
			\left. \cos \varphi \right|_{\varphi=0} +
		\int_{\theta_\x}^{\thetamax}
			\left. \cos \varphi \right|_{\varphi=\varphi_\x} -
		\int_{\theta_\y}^\thetamax
			\left. \cos \varphi \right|_{\varphi=\varphi_\y}
	\right)
		\sin ^2\theta 
	\dd\theta = \\
	\x\z	
	\left(
		\int_\thetamin^{\theta_\x}
			1 +
		\int_{\theta_\x}^{\thetamax}
			\dfrac{a}{t\sin\theta} -
		\int_{\theta_\y}^\thetamax
			\sqrt{1-\dfrac{b^2}{t^2\sin^2\theta}}
	\right)
		\sin ^2\theta 
	\dd\theta = \\
	\x\z	
	\left(
		\int_\thetamin^{\theta_\x}
			\sin^2\theta\dd\theta +
		\int_{\theta_\x}^{\thetamax}
			\dfrac{a\sin\theta}{t}\dd\theta -
		\int_{\theta_\y}^\thetamax
			\sqrt{\sin^2\theta-\dfrac{b^2}{t^2}}\sin\theta\dd\theta
	\right)
	= \\
	\x\z	
	\left(
		\frac12\left[
			\theta-\sin\theta\cos\theta
		\right]_\thetamin^{\theta_\x} +
		\left[
			\dfrac{-a\cos\theta}{t}
		\right]_{\theta_\x}^{\thetamax} -
		\int_{\theta_\y}^\thetamax
			\sqrt{1-\dfrac{b^2}{t^2} - \cos^2\theta }\sin\theta \dd\theta
	\right) \label{f_explicit}
\end{gather}
where the last integral inside the parentheses may be written as
\begin{gather}
		\left[
			-\frac{1}{2} \left(\cos\theta \sqrt{1-\dfrac{b^2}{t^2}-\cos^2\theta}+\left(1-\dfrac{b^2}{t^2}\right) \atan \left(\frac{\cos\theta}{\sqrt{1-\dfrac{b^2}{t^2}-\cos^2\theta}}\right)\right)
		\right]_{\theta_\y}^\thetamax = \\
		\left[
			-\frac{1}{2} \left(\cos\theta \sqrt{\sin^2\theta-\dfrac{b^2}{t^2}}+\left(1-\dfrac{b^2}{t^2}\right) \atan \left(\frac{\cos\theta}{\sqrt{\sin^2\theta-\dfrac{b^2}{t^2}}}\right)\right)
		\right]_{\theta_\y}^\thetamax \label{hideous}
\end{gather}
whenever $\theta_\y<\pi/2$, by using the fact that $\frac{1}{2} \left(x \sqrt{c-x^2}+c \atan \left(\frac{x}{\sqrt{c-x^2}}\right)\right)$ is an antiderivative of $\sqrt{c-x^2}$ with respect to $x$ when $c$ is a constant.
We obtain $f(\y,\z,\x)$ and $f(\z,\x,\y)$ by switching the roles of $\x,\y,\z$ in \eqref{f_explicit}.

It remains to insert the limits $\thetamin,\theta_\x,\theta_\y,\thetamax$ into the antiderivatives \eqref{h_explicit}, \eqref{g_explicit} and \eqref{f_explicit} above. Noting that $\thetamin,\theta_\x,\theta_\y,\thetamax$ are expressed in terms of piecewise-defined functions, the following manipulations will be useful.
For any function $\psi$, we have
\begin{gather}
  \psi(\thetamin)
  =\psi\left(\cos\inv\frac\z t\right)\chi_\z + \psi(\cos\inv 1)(1-\chi_\z) \\
  =\left(\psi\left(\cos\inv\frac\z t\right)-\psi(0)\right)\chi_\z + \psi(0)
\end{gather}
where $\chi_\z\defeq \chi(t>\z)$. Similarly,
\begin{gather}
  \psi(\thetamax)
  =\left(\psi\left(\sin\inv \frac{\sqrt{\x^2+\y^2}}t \right)-\psi(\pi/2)\right)\chi_{\x,\y} +
   \psi(\pi/2)
\end{gather}
where $\chi_{\x,\y}\defeq \chi(\sqrt{\x^2+\y^2}>t)$, and
\begin{gather}
  \psi(\theta_\x)=
  (1-\chi_\x)\psi(\pi/2) +
  (\chi_\x-\chi_{\x,\z}) \psi\left(\sin\inv\frac\x t\right) +
  \chi_{\x,\z} \psi\left(\cos\inv\frac\z t\right) \\
  = \chi_{\x,\z}\cdot\left(
    \psi\left(\cos\inv\frac\z t\right)-\psi\left(\sin\inv\frac\x t\right)
  \right)+
  \chi_\x\cdot\left(
    \psi\left(\sin\inv\frac\x t\right)-\psi(\pi/2)
  \right)+
  \psi(\pi/2)
\end{gather}
and similarly, $\psi(\theta_\y)$ can be written as
\begin{gather}
  \chi_{\y,\z}\cdot\left(
    \psi\left(\cos\inv\frac\z t\right)-\psi\left(\sin\inv\frac\y t\right)
  \right)+
  \chi_\y\cdot\left(
    \psi\left(\sin\inv\frac\y t\right)-\psi(\pi/2)
  \right)+
  \psi(\pi/2).
\end{gather}

With this we can evaluate $[\psi]_\thetamin^{\theta_\x}, [\psi]_{\theta_\x}^\thetamax, [\psi]_{\theta_\y}^\thetamax$. But since we know that we will get a function symmetric with respect to the values $\x,\y,\z$, it suffices to keep only those terms with $\chi_\x$ and $\chi_{\x,\y}$, say, and then the other terms may be evaluated by just switching the order of $\x,\y,\z$. 
Upon inserting the limits and differentiating, one obtains (after tedious calculations) that 
  \begin{gather}
    \operatorname{pdf}_X(t) = \dfrac{F(\x,\y,\z,t) + F(\y,\z,\x,t) + F(\z,\x,\y,t)}{3 \pi  t^3 (\x\y+\x\z+\y\z)}
  \end{gather}
  where
  \begin{gather}
    \begin{align}
    F(\x,\y,\z,t) \defeq\ &(8\x t^3-3t^4) + \\
    \chi(t\geq \x)\Bigg(\left(6t^4-\x^4+6 \pi  \x^2 \y \z\right)-&(8 \x t^3-3t^4)-4(\y+\z)\sqrt{\abs{t^2-\x^2}}(\x^2+2t^2)\Bigg)+
    \end{align}\\
    \begin{align}
      \chi(t\geq\sqrt{\x^2+\y^2})\Bigg[
       & \x^4+\y^4-9 t^4-6 \x^2 \y^2+\sqrt{\abs{t^2-\x^2-\y^2} } 4\z\left( \x^2 + \y^2 +2 t^2\right)+\\
       & 4\x\sqrt{\abs{t^2-\y^2}}(\y^2+2t^2)-12 \x^2 \y \z \cdot \arctan\left(\dfrac {\sqrt{\abs{t^2-\x^2-\y^2} }}\y\right)+\\
       & 4\y\sqrt{\abs{t^2-\x^2}}(\x^2+2t^2)-12 \x \y^2 \z \cdot \arctan\left(\frac{\sqrt{\abs{ t^2-\x^2-\y^2} }}{\x}\right) \Bigg].
    \end{align}
  \end{gather}
Rewriting $F$ as a piecewise function, we get Theorem \eqref{thm_3}.

\section{Proof of Theorem \ref{thm_hard}}
Consider the distribution of the random variable
$Y_{M,N}$.  Since we record the {\em same} number of bounces
for each choice of angle $\varphi$ we may replace the $M$-particle system
with a one particle system $Y_{N}$ as follows: randomly select, with
uniform distribution, the angle $\varphi$ and generate $N$ bounce lengths
and randomly select one of these bounce lengths (with uniform
distribution); by the strong law of large numbers, $Y_{M,N}$ converges
in distribution 
to $Y_{N}$ as $M\to\infty$.

We now determine the limit distribution of $Y_{N}$.
As before, we first unfold the motion, and replace motion in a box
with specular reflections on the walls with motion in $\R^2$; see
Figure \ref{unfold_rgb}.  The path lengths between bounces is then
the same as the lengths between the intersections with horizontal or
vertical grid lines.  To understand the spatial distribution, we
project the dynamics to the torus $\R^2/\Lambda$ where $\Lambda$ is
the lattice
\begin{gather}
\Lambda = \{  (n_1\x,n_2\y) : n_1,n_2 \in \Z \},
\end{gather}
and we may identify the torus with the rectangle $[0,\x]\times[0,\y]$.

Let us first consider the motion of a single particle with an arbitrary
initial position, and direction of motion given by an angle $\varphi$.
Taking symmetries into account, we may assume
that $\varphi \in [0,\pi/2]$.  (Note that $\frac{d \varphi}{\pi/2}$
gives a probability measure on these angles.)
If the particle travels a large distance $R>0$, the number of intersections with
horizontal, respectively vertical, grid lines is $\frac{R \sin
  \varphi}{\y} +O(1)$, respectively $\frac{R \cos \varphi}{\x} +O(1)$.
Thus, in the limit $R \to \infty$, the probability of a line segment
beginning at a horizontal (respectively vertical) grid line is given
by $P_{h}$, respectively $P_{v}$ (here we suppress the dependence on
$\varphi$) where
\begin{gather}
P_{h} \defeq \frac{\frac{\sin \varphi}{\y}}{\frac{\sin \varphi}{\y}+\frac{\cos
  \varphi}{\x}}, \quad
P_{v} \defeq \frac{\frac{\cos \varphi}{\x}}{\frac{\sin \varphi}{\y}+\frac{\cos
  \varphi}{\x}}.
\end{gather}
The unfolded flow on the torus is ergodic for almost all $\varphi$,
and thus the starting points of the line segments becomes uniformly
distributed as $R\to\infty$ for almost all $\varphi$.

Let 
\begin{gather}
T = T(\varphi) \defeq \x/\cos \varphi.
\end{gather}
Since $\sin \varphi = \sqrt{T^2-\x^2}/T$, we obtain that
\begin{equation}
  \label{eq:vert-horiz-probs}
P_{h} = \frac{\sqrt{T^2-\x^2}}{\y+\sqrt{T^2-\x^2}}, 
\quad
P_{v} = \frac{\y}{\y+\sqrt{T^2-\x^2}}.
\end{equation}

Let $\theta = \arctan \y/\x$ denote the angle of the diagonal in the
box, and assume that $0 \leq \varphi \leq \theta$.  We then
observe the following regarding the line segment lengths.

First, if the segment begins at a horizontal line, it must end at a
vertical line, and the possible lengths of these segment lie between
$0$ and $T$.  We find that these lengths are uniformly distributed in
$[0,T]$ since the starting points of the segments are uniformly
distributed.

On the other hand, if the line segment begins at a vertical line, it
can either end at a vertical or horizontal line.  Since the
starting points are uniformly distributed, the former happens with probability
\begin{gather}
\frac{\x \tan \varphi}{\y} = 
\frac{\x \frac{\sqrt{T^{2}-\x^{2}}}{\x}}{\y}  = 
\frac{\sqrt{T^{2}-\x^{2}}}{\y}  
\end{gather}
and the length of the segment is again uniformly distributed in
$[0,T]$, whereas the latter happens with probability 
\begin{gather}
\frac{\y - \x \tan \varphi}{\y} = 
1- \frac{\sqrt{T^{2}-\x^{2}}}{\y}  
\end{gather}
in which case the segment is always of length $T$.  

Now, $\varphi \in [0,\theta]$ implies that $T \in
[\x,\sqrt{\x^{2}+\y^{2}}]$, and noting that 
\begin{gather}
\frac{d \varphi}{\dd T} = \frac{\x}{T \sqrt{T^{2}-\x^{2}}}
\end{gather}
we find that the probability of observing a line segment of length $t$
is the sum of a ``singular part'' (the segment begins and ends on
vertical lines; note that all such segments have the {\em same}
lengths) and a ``smooth part'' (the segment does not begin and end on
vertical lines).  Moreover, the smooth part contribution equals
\begin{gather}
\frac{1}{\pi/2}
\int_{\max(\x,t)}^{\sqrt{\x^{2}+\y^{2}}}
\frac{1}{T}
\left(
P_{h} + P_{v} \frac{\x \tan \varphi}{\y}
\right) \frac{d \varphi}{\dd T} \dd T
\end{gather}
which, on inserting (\ref{eq:vert-horiz-probs}), equals
\begin{gather}
  \frac{1}{\pi/2}
  \int_{\max(\x,t)}^{\sqrt{\x^{2}+\y^{2}}}
  \frac{1}{T} \cdot
  \left(
  \frac{\sqrt{T^2-\x^2}}{\y+\sqrt{T^2-\x^2}} + 
  \frac{\y}{\y+\sqrt{T^2-\x^2}} \frac{\x \tan \varphi}{\y}
  \right) \cdot \frac{\x}{T \sqrt{T^{2}-\x^{2}}} \dd T
  =\\
  \frac{1}{\pi/2}
  \int_{\max(\x,t)}^{\sqrt{\x^{2}+\y^{2}}}
  \frac{1}{T} \cdot
  \left(
  \frac{\sqrt{T^2-\x^2}}{\y+\sqrt{T^2-\x^2}} + 
  \frac{\y}{\y+\sqrt{T^2-\x^2}} \frac{\sqrt{T^{2}-\x^{2}}}{\y}  
  \right) \cdot \frac{\x}{T \sqrt{T^{2}-\x^{2}}} \dd T
  =\\
  \frac{1}{\pi/2}
  \int_{\max(\x,t)}^{\sqrt{\x^{2}+\y^{2}}}
  \frac{2\x}{\y+\sqrt{T^2-\x^2}} 
  \cdot \frac{\dd T}{T^{2} } .
\end{gather}

On the other hand, the ``singular part contribution'', provided $t
\geq \x$, to the probability of a segment having length $t$ equals
\begin{gather}
\frac{ P_{v}}{\pi/2} 
 \cdot \frac{\y- \x \tan \varphi}{\y} \cdot \frac{d \varphi}{dt} 
=
\frac{1}{\pi/2} \cdot
\frac{\y}{\y+\sqrt{t^2-\x^2}} \cdot
\left(1- \frac{\sqrt{t^{2}-\x^{2}}}{\y}   \right) \cdot
\frac{\x}{t \sqrt{t^{2}-\x^{2}}}
=\\
\frac{1}{\pi/2} \cdot
\frac{\x}{t  (\y+\sqrt{t^2-\x^2}) \sqrt{t^{2}-\x^{2}} } \cdot
\left(\y - \sqrt{t^{2}-\x^{2}}   \right).
\end{gather}

In case $\theta \leq \varphi \leq \pi/2$, a similar argument (we simple
reverse the roles of $\x$ and $\y$) shows that the smooth contribution
equals
\begin{gather}
\frac{1}{\pi/2}
\int_{\max(\y,t)}^{\sqrt{\x^{2}+\y^{2}}}
\frac{2\y}{\x+\sqrt{T^2-\y^2}} 
\cdot \frac{\dd T}{T^{2} } 
\end{gather}
and that the singular contribution (if $t \geq \y$) equals
\begin{gather}
\frac{1}{\pi/2} \cdot
\frac{\y}{t  (\x+\sqrt{t^2-\y^2}) \sqrt{t^{2}-\y^{2}} } \cdot
\left(\x- \sqrt{t^{2}-\y^{2}}   \right) .
\end{gather}

Thus, if we let $P_{\op{sing}}(t)$ denote the ``singular
contribution'' to the probability density function we find the following: if $t < \x$, then
\begin{gather}
P_{\op{sing}}(t) = 0
\end{gather}
if $t \in [\x,\y]$, then
\begin{gather}
P_{\op{sing}}(t) = 
\frac{1}{\pi/2} \cdot 
\frac{\x \left(\y- \sqrt{t^{2}-\x^{2}}  \right) }
{t  (\y+\sqrt{t^2-\x^2}) \sqrt{t^{2}-\x^{2}} } 
\end{gather}
and if $t \in [\y, \sqrt{\x^{2}+\y^{2}}]$, then
\begin{gather}
P_{\op{sing}}(t)= 
\frac{1}{\pi/2} \cdot \left(
\frac{\x(\y- \sqrt{t^{2}-\x^{2}})}{t  (\y+\sqrt{t^2-\x^2}) \sqrt{t^{2}-\x^{2}} } 
+
\frac{\y(\x- \sqrt{t^{2}-\y^{2}})}{t  (\x+\sqrt{t^2-\y^2}) \sqrt{t^{2}-\y^{2}} } 
\right).
\end{gather}
\begin{Remark}
\label{rem:singularity-explained}
  Note that $P_{\op{sing}}$ has a singularity of type
  $(t-\x)^{-1/2}$ just to the right of $t=\x$ (and similarly just to the right of $t=\y$).
  In a sense this singularity arises from
  the singularity in the change of variables $\varphi \mapsto T$ since
  $\frac{d \varphi}{dT} = \frac{\x}{T \sqrt{T^{2}-\x^{2}}}$.
  The reason for the singularities in the spreading model for $n=2$ is
  similar, as the spreading model can be obtained from the absorption model by a smooth
  change of the angular measure.
\end{Remark}

Similarly, the ``smooth part'' of the contribution is (for
$t \in [0,\sqrt{\x^{2}+\y^{2}}]$) given by
\begin{gather}
P_{\op{smooth}}(t) =
\frac{1}{\pi/2}
\left(
\int_{\max(\x,t)}^{\sqrt{\x^{2}+\y^{2}}}
\frac{2\x}{\y+\sqrt{T^2-\x^2}} 
\cdot \frac{\dd T}{T^{2} } 
+
\int_{\max(\y,t)}^{\sqrt{\x^{2}+\y^{2}}}
\frac{2\y}{\x+\sqrt{T^2-\y^2}} 
\cdot \frac{\dd T}{T^{2} } 
\right)
\end{gather}

Hence the probability density function of the distribution of the segment length $t$ is given by 
\begin{gather}
  \op{pdf}_Y(t) = P_{\op{sing}}(t) +  P_{\op{smooth}}(t).
\end{gather}

We will now evaluate $P_{\op{smooth}}(t)$.
An antiderivative of $\frac{2\x}{\y+\sqrt{T^2-\x^2}} \cdot \frac{1}{T^{2} } $ with respect to $T$ for $T\in(\x,\sqrt{\x^2+\y^2})$ is
\begin{gather}
  \frac{2 \x( \sqrt{T^2-\x^2}- \y)}{T \left(\x^2+\y^2\right)}+\frac{2 \x \y \left(\tanh
   ^{-1}\left(\frac{T}{\sqrt{\x^2+\y^2}}\right)-\tanh ^{-1}\left(\frac{\sqrt{T^2-\x^2}
   \sqrt{\x^2+\y^2}}{T \y}\right)\right)}{\left(\x^2+\y^2\right)^{3/2}} \label{an_explicit_integrand:1}
\end{gather}
where $\tanh\inv(z)=\frac12\log\frac{1+z}{1-z}$ for $\abs z<1$. (A quick calculation shows that $\frac{\sqrt{T^2-\x^2}\sqrt{\x^2+\y^2}}{T \y} < 1$ whenever $\x<T<\sqrt{\x^2+\y^2}$.) We can rewrite \eqref{an_explicit_integrand:1} as
\begin{gather}
  \frac{2 \x( \sqrt{T^2-\x^2}- \y)}{T \left(\x^2+\y^2\right)}+\frac{\x \y 
    \log \left(\frac{\left(\sqrt{\x^2+\y^2}+T\right) \left(T \y-\sqrt{T^2-\x^2}
   \sqrt{\x^2+\y^2}\right)}{\left(\sqrt{\x^2+\y^2}-T\right) \left(T \y + \sqrt{T^2-\x^2} \sqrt{\x^2+\y^2}\right)}\right)
  }{\left(\x^2+\y^2\right)^{3/2}} \label{an_explicit_integrand:2}
\end{gather}
By l'Hôpital's rule we have
\begin{gather}
  \lim_{T\to\sqrt{\x^2+\y^2}+}
  \frac{T \y-\sqrt{T^2-\x^2}\sqrt{\x^2+\y^2}} {\sqrt{\x^2+\y^2}-T } = 
  \lim_{T\to\sqrt{\x^2+\y^2}+}
  \frac{\y-\frac{T}{\sqrt{T^2-\x^2}}\sqrt{\x^2+\y^2}} {-1 } =
  \frac{\x^2}\y
\end{gather}
so the limit of \eqref{an_explicit_integrand:1} as $T\to\sqrt{\x^2+\y^2}+$ is
\begin{gather}
  \frac{\x \y 
    \log \left(
      \left(\frac{\x^2}\y\right)\cdot
    \frac{\left(\sqrt{\x^2+\y^2}+\sqrt{\x^2+\y^2}\right) }{\left(\y \sqrt{\x^2+\y^2}+\y \sqrt{\x^2+\y^2}\right)}\right)
  }{\left(\x^2+\y^2\right)^{3/2}} =
  \frac{2\x \y 
    \log \left( \frac{\x} {\y}\right)
  }{\left(\x^2+\y^2\right)^{3/2}}.
\end{gather}
The limit of \eqref{an_explicit_integrand:1} as $T\to \x+$ is
\begin{gather}
  \frac{-2 \y}{ \left(\x^2+\y^2\right)}+\frac{2 \x \y \tanh^{-1}\left(\frac{\x}{\sqrt{\x^2+\y^2}}\right)}{\left(\x^2+\y^2\right)^{3/2}}.
\end{gather}
Thus, assuming $\x<\y$, we can write $\frac\pi 2P_{\op{smooth}}(t)$ as
\begin{gather}
      \frac{2 (\x+\y)}{ \left(\x^2+\y^2\right)}-\frac{2 \x \y }{\left(\x^2+\y^2\right)^{3/2}}\left(\tanh^{-1}\left(\frac{\x}{\sqrt{\x^2+\y^2}}\right) + \tanh^{-1}\left(\frac{\y}{\sqrt{\x^2+\y^2}}\right)\right)
\end{gather}
if $t<\x,\y$, or as
\begin{gather}
  \frac{2 \x \y +2\x t - 2\x\sqrt{t^2-\x^2}}{t \left(\x^2+\y^2\right)}+\\
  \frac{2 \x \y \left(-\tanh
   ^{-1}\left(\frac{t}{\sqrt{\x^2+\y^2}}\right)+\tanh ^{-1}\left(\frac{\sqrt{t^2-\x^2}
   \sqrt{\x^2+\y^2}}{t \y}\right) - \tanh^{-1}\left(\frac{\y}{\sqrt{\x^2+\y^2}}\right)\right)}{\left(\x^2+\y^2\right)^{3/2}}
\end{gather}
if $\x<t<\y$ or as
\begin{gather}
  2\frac{2\x\y- \x\sqrt{t^2-\x^2} - \y \sqrt{t^2-\y^2}}{t \left(\x^2+\y^2\right)}+ \\
   \frac{2 \x \y \left(-2\tanh
   ^{-1}\left(\frac{t}{\sqrt{\x^2+\y^2}}\right)+\tanh ^{-1}\left(\frac{\sqrt{t^2-\x^2}
   \sqrt{\x^2+\y^2}}{t \y}\right) +
    \tanh ^{-1}\left(\frac{\sqrt{t^2-\y^2}
       \sqrt{\x^2+\y^2}}{t \x}\right)
    \right)}{\left(\x^2+\y^2\right)^{3/2}}
\end{gather}
if $\x,\y<t$. Adding $P_{\op{sing}}(t)$ to this, we get Theorem \ref{thm_hard}.


\appendix

\section{Calculation of an integral}

\begin{Lemma} \label{integral_of_vn}
Write $|\mathbb S^{n-1}|$ for the $(n-1)$-dimensional surface area of the sphere
$\mathbb S^{n-1}\subseteq \R^n$. Then we have 
  \begin{gather}
    \int_{\mathbb S^{n-1}_+} v_n \dd S(v) = \dfrac1{\pi}\dfrac{|\mathbb S^n|}{2^{n}}.
  \end{gather}
where $\mathbb S^{n-1}_+ \defeq \mathbb S^{n-1}\cap (0,\infty)^n$ is the part of the sphere $\mathbb S^{n-1}$ with positive coordinates.
\end{Lemma}
\begin{proof}
  We may parametrize $v=(v_1,\ldots,v_n)\in \mathbb S^{n-1}_+$ with
  \begin{align}
    v_1 &= \cos \theta_1 \\
    v_2 &= \sin\theta_1 \cos\theta_2 \\
    v_3 &= \sin\theta_1 \sin\theta_2 \cos \theta_3 \\
     &\vdots \\
    v_{n-1} &= \sin\theta_1\cdots \sin\theta_{n-2}\cos\theta_{n-1} \\
    v_n &= \sin\theta_1\cdots \sin\theta_{n-2}\sin\theta_{n-1}
  \end{align}
  for $\theta_1,\ldots,\theta_{n-1}\in(0,\pi/2)$. We have the spherical area element
  \begin{gather}
    \dd S(v) = \sin^{n-2}\theta_1 \sin^{n-3} \theta_2 \cdots \sin \theta_{n-2} \dd{\theta_1}\cdots \dd{\theta_{n-1}}.
  \end{gather}
  Thus we get
  \begin{gather}
    \int_{\mathbb S^{n-1}_+} v_n \dd S(v) = 
    \prod_{i=1}^{n-1} \int_0^{\pi/2} \sin^{n-1-i} \theta_i \dd{\theta_i}.
  \end{gather}
  Introducing an additional integration variable $\theta_n$, we recognize the integrand as the spherical area element in $n+1$ dimensions, and thus the above is
  \begin{gather}
    \dfrac1{\int_0^{\pi/2}\dd{\theta_n}} \prod_{i=1}^{n} \int_0^{\pi/2} \sin^{n-1-i} \theta_i \dd{\theta_i} =
    \dfrac1{\pi/2}\dfrac{|\mathbb S^n|}{2^{n+1}}.
  \end{gather}
  since $\int_{\mathbb S^{n}_+} \dd S(v) = |\mathbb S^n|/2^{n+1}$.
\end{proof}


\end{document}